\documentclass[12pt]{article}
\usepackage{amsfonts,amssymb,amsmath,bm,amsthm}

\newtheorem{theorem}{Theorem}

\newtheorem{corollary}[theorem]{Corollary}
\newtheorem{lemma}[theorem]{Lemma}
\newtheorem{proposition}{Proposition}

\newtheorem{conjecture}[theorem]{Conjecture}
\newtheorem{theoreml}{Theorem L}
\newtheorem{theoremBZ}{Theorem BZ}

\newtheorem{theorembhwy}{Theorem BHWY}

\theoremstyle{definition}
\newtheorem{remark}[theorem]{Remark}

\newtheorem{notation}[theorem]{Notation}

\def\R{\mathbb{R}}

\def\Q{\mathbb{Q}}
\def\Z{\mathbb{Z}}
\def\N{\mathbb{N}}

\newcommand{\ve}{\mathbf{e}}

\newcommand{\vv}[1]{{\mathbf{#1}}}

\renewcommand{\tilde}{\widetilde}

\newcommand{\bigk}{\mathop{\mathbf{K}}}

\newcommand{\hidden}[1]{}

\newcommand{\ftm}{f_{TM}}

\newcommand{\cTM}{\tau_{TM}}

\newcommand{\va}{\mathbf{a}}
\newcommand{\vb}{\mathbf{b}}
\newcommand{\vc}{\mathbf{c}}
\newcommand{\vu}{\mathbf{u}}

\newcommand{\logc}{\log_2}

\title{On irrationality measure of Thue-Morse constant}

\author{Dzmitry Badziahin\footnote{Univerity of Sydney, School of Mathematics and Statistics, NSW 2006} \;and Evgeniy Zorin\footnote{University of York, Department of Mathematics, York, YO10 5DD, United Kingdom.}\thanks{
Evgeniy Zorin acknowledges the support of EPSRC  Grant EP/M021858/1}
}

\begin{document}

\maketitle

\begin{abstract}
We provide a non-trivial measure of irrationality for a class of Mahler
numbers defined with infinite products which cover the Thue-Morse constant.
Among the other things, our results imply a generalization
to~\cite{bugeaud_2011}.
\end{abstract}

\section{Introduction}

Let $\xi\in\R$ be an irrational number. Its irrationality exponent $\mu(\xi)$
is defined to be the supremum of all $\mu$ such that the inequality
$$
\left|\xi - \frac{p}{q}\right|<q^{-\mu}
$$
has infinitely many rational solutions $p/q$. This is an important property
of a real number since it shows, how close the given real number can be
approximated by rational numbers in terms of their denominators. The
irrationality exponent can be further refined by the following notion.
%The following refinement of $\mu(\xi)$ can be used to further distinguish
%between numbers with the same irrationality exponent.
%At the same time, it is well known that among the real numbers with the same
%irratoinality exponent, for instance 2, some numbers admits better rational
%approximations and some others are worse approximable, in the sense which can
%be defined with the help of the following refinement of $\mu(\xi)$. \hidden{
%Therefore it becomes a natural question to distinguish between numbers with
%the same irrationality exponent. It can be done with the help of the
%following refinement of $\mu(\xi)$. }
Let $\psi(q):\R_{\ge 0} \to \R_{\ge 0}$ be a function which tends to zero as
$q\to \infty$. Any function $\psi$ with these properties is referred to as the
\emph{approximation function}. We say that an irrational number $\xi$ is
\emph{$\psi$-well approximable} if the inequality
\begin{equation}\label{def_well}
\left|\xi -
\frac{p}{q}\right|<\psi(q)
\end{equation}
has infinitely many solutions $p/q\in\Q$. Conversely, we say that $\xi$ is
\emph{$\psi$-badly approximable} if~\eqref{def_well} has only finitely many
solutions. Finally, we say that $\xi$ is \emph{badly approximable} if it is
$c/q$-badly approximable for some positive costant $c>0$.

If a number $\xi\in\R$ is $\psi$-badly approximable, we also say that $\psi$
is a \emph{measure of irrationality of $\xi$}.

The statement $\mu(\xi) = \mu$ is equivalent to saying that for any
$\epsilon>0$, $\xi$ is both $q^{-\mu-\epsilon}$-well approximable and
$q^{-\mu+\epsilon}$-badly approximable. On the other hand, $(q^2\log
q)^{-1}$-badly approximable numbers are in general worse approached by
rationals when compared to $(q^2\log^2 q)^{-1}$-badly approximable numbers,
even though that both of them have irrationality exponent equal to 2.

\begin{remark} \label{im_same_c}
It is quite easy to verify
%, straightforwardly with Definition~\ref{def_im},
that, for any approximation function $\psi$, for any $\xi\in\R$ and any
$c\in\Q\setminus\{0\}$, the numbers $\xi$ and $c\xi$ simultaneously are or
are not $\psi$-badly approximable. Similarly, they simultaneously are or are
not $\psi$-well approximable.
\end{remark}

A big progress has been made recently in determining Diophantine
approximation properties of so called Mahler numbers. Their definition
slightly varies in the literature. In the present paper we define Mahler
functions and Mahler numbers as follows. An analytic function $F(z)$ is
called \emph{Mahler function} if it satisfies the functional equation
\begin{equation}\label{def_mahlf}
\sum_{i=0}^n P_i(z)F(z^{d^i}) = Q(z)
\end{equation}
where $n$ and $d$ are positive integers with $d\ge 2$, $P_i(z), Q(z) \in
\Q[z]$, $i=0,\dots,n$ and $P_0(z)P_n(z) \neq 0$. We will only consider those
Mahler functions $F(z)$ which lie in the space $\Q((z^{-1}))$ of Laurent
series. Then, for any $\alpha\in\overline{\Q}$ inside the disc of convergence
of $F(z)$, a real number $F(\alpha)$ is called a \emph{Mahler number}.

One of the classical examples of Mahler numbers is the so called Thue-Morse
constant which is defined as follows. Let $\vv
t=(t_0,t_1,\dots)=(0,1,1,0,1,0,0,\dots)$ be the Thue-Morse sequence, that is
the sequence $(t_n)_{n\in\N_0}$, where $\N_0:=\N\cup\{0\}$, defined by the
recurrence relations $t_0=0$ and for all $n\in\N_0$
$$
\begin{aligned}
t_{2n}&=t_n,\\
t_{2n+1}&=1-t_n.
\end{aligned}
$$
Then, the Thue-Morse constant $\cTM$ is a real number which binary expansion
is the Thue-Morse word. In other words,
\begin{equation} \label{defTM}
\cTM:=\sum_{k=0}^{\infty}\frac{t_k}{2^{k+1}}.
\end{equation}

It is well known that $\cTM$ is a Mahler number. Indeed, one can check that
$\cTM$ is related with the generating function
\begin{equation} \label{ftm_presentation}
\ftm(z):=\sum_{i=0}^{\infty}(-1)^{t_i}z^{-i}
\end{equation}
by the formula $\cTM = \frac12 (1- \frac12 \ftm(2))$. At the same time, the function $\ftm(z)$, defined
by~\eqref{ftm_presentation}, admits the following presentation~\cite[\S13.4]{AS2003}:
$$
\ftm(z)=\prod_{k=0}^{\infty}\left(1-z^{-2^k}\right),
$$
and the following functional equation holds:
\begin{equation} \label{func_eq}
\ftm(z^2)=\frac{z}{z-1}\ftm(z).
\end{equation}
So it is indeed a Mahler function.

Approximation of Mahler numbers by algebraic numbers has been studied within
a broad research direction on transcendence and algebraic independence of
these numbers. We refer the reader to the monograph~\cite{Ni1996} for more
details on this topic. %%%MORE RECENT RESULTS AFTER 1996? (Bugeaud, Adamczewski for example)
%Because of a large number of results accumulated on this
%subject by the present time, maybe a better choice would be to refer the
%interested reader to a nice monograph on this topic~\cite{Ni1996}.

It has to be mentioned that, though some results on approximation by
algebraic numbers can be specialized to results on rational approximations,
most often they become rather weak. This happens because the results on
approximations by algebraic numbers necessarily involve complicated
constructions, which results in some loss of precision. More fundamental
reason is that rational numbers enjoy significantly more regular (and much
better understood) distribution in the real line when compared to the
algebraic numbers.

The history of the research of approximation properties of Mahler numbers by
rational numbers probably started in the beginning of 1990th with the work of
Shallit and van der Poorten~\cite{vdP_S}, where they considered a class of
numbers that contains some Mahler numbers, including Fredholm constant
$\sum_{n=0}^{\infty}10^{-2^n}$, and they proved that all numbers from that
class are badly approximable.

The next result on the subject, the authors are aware of, is due to
Adamczewski and Cassaigne. In 2006, they
proved~\cite{adamczewski_cassaigne_2006} that every automatic number (which,
according to~\cite[Theorem 1]{becker_1994}, is a subset of Mahler numbers)
has finite irrationality exponent, or, equivalently, every automatic number
is not a Liouville number. Later, this result was extended to all
Mahler numbers~\cite{bell_bugeaud_coons_2015}. % and eventually to all Mahler numbers~\cite{zorin_2017}.
We also mention here the result by Adamczewski and
Rivoal~\cite{adamczewski_rivoal_2009}, where they showed that some classes of
Mahler numbers are $\psi$-well approximable, for various functions $\psi$
depending on a class under consideration.

The Thue-Morse constant is one of the first Mahler numbers which irrationality
exponent was computed precisely, % Formally speaking, the irrationality exponent of Fredholm constant was computed before
it has been done by Bugeaud in 2011~\cite{bugeaud_2011}. This result served as a foundation
for several other works, establishing precise values of irrationality
exponents for wider and wider classes of Mahler numbers, see for
example~\cite{coons_2013, guo_wu_wen_2014, wu_wen_2014}.

Bugeaud, Han, Wen and Yao~\cite{bugeaud_han_wen_yao_2015} computed the
estimates of $\mu(f(b))$ for a large class of Mahler functions $f(z)$,
provided that the distribution of indices at which Hankel determinants of
$f(z)$  do not vanish or, equivalently, the continued fraction of $f(z)$ is
known.  In many cases, these estimates lead to the precise value of
$\mu(f(b))$. We will consider this result in more details in the next
subsection. Later, Badziahin~\cite{badziahin_2017} provided a continued
fraction expansion for the functions of the form
$$
f(z) = \prod_{t=0}^\infty P(z^{-d^t})
$$
where $d\in\N, d\ge 2$ and $P(z)\in \Q[z]$ with $\deg P< d$. This result, complimented with~\cite{bugeaud_han_wen_yao_2015}, allows to find sharp estimates for the values of these functions at integer points.

Despite rather extensive studies on irrationality exponents of Mahler
numbers, very little is known about their sharper Diophantine approximation
properties. In 2015, Badziahin and Zorin~\cite{BZ2015} proved that the
Thue-Morse constant $\cTM$, together with many other values of $\ftm(b)$,
$b\in\N$, are not badly approximable. Moreover, they proved
%(actually, for a family of its straightforward generaliations of the Thue-Morse constant)
\begin{theoremBZ}
Thue-Morse constant $\cTM$ is $\frac{C}{q^2\log\log q}$-well approximable, for some explicit constant $C>0$.
\end{theoremBZ}
Later, in~\cite{badziahin_zorin_2015} they extended this result to the values
$f_3(b)$, where $b$ is from a ceratin subset of positive integers, and
$$
f_3(z):= \prod_{t=0}^\infty (1 - z^{-3^t}).
$$

Khintchine's Theorem implies that outside of a set of the Lebesgue measure
zero, all real numbers are $\frac{1}{q^2\log q}$-well approximable and
$\frac{1}{q^2\log^2 q}$-badly approximable. Of course, this metric result
implies nothing for any particular real number, or countable family of real
numbers. However, it sets some expectations on the Diophantine approximaiton
properties of real numbers. %Morally, it is not unreasonable to expect that if a number has no ``explicit'' reasons to

The result of Theorem~BZ does not provide the well-approximability result for
the Thue-Morse constant suggested by Khintchine's theorem, but it falls
rather short to it. At the same time, the bad-approximability side, suggested
by Khintchine theorem, seems to be hard to establish (or even to approach to
it) in the case of Thue-Morse constant and related numbers. In this paper we
prove that a subclass of Mahler numbers, containing, in particular,
Thue-Morse constant, is $(q\exp(K\sqrt{\log q\log\log q}))^{-2}$-badly
approximable for some
constant $K>0$, see Theorem~\ref{th_main} %and Corollary~\ref{cor_main}
 at the end of Subsection~\ref{ss_CF_LS}.
This result is still pretty far from what is suggested by Khintchine's
theorem, however it significantly improves the best
result~\cite{bugeaud_2011} available at this moment, namely, that the
irrationality exponent of Thue-Morse constant equals 2.

%We give the precise statement of our main result in Theorem~\ref{th_main}, at the end of Subsection~\ref{ss_CF_LS}, where we explain necessary background material.

%The main result of this paper, Theorem~\ref{th_main}, % and Corollary~\ref{cor_main}, are
%is presented at the end of Subsection~\ref{ss_CF_LS}, where we explain more on our motivation for these studies.
\hidden{
\begin{theorem}\label{th_main}
Let $d\ge 2$ be an integer and
\begin{equation}\label{def_f}
f(z) = \prod_{t=0}^\infty P(z^{-d^t}),
\end{equation}
where $P(z)\in \Z[z]$ is a polynomial such that $P(1) = 1$ and $\deg P(z)<d$.
Assume that the series $f(z)$ is badly approximable (i.e. degrees of all
partial quotients of $f(z)$ are bounded from above by an absolute constant).
Then there exists a positive constant $C$ such that for any $b\in\Z$, $b\geq 2$, we have either $f(b) = 0$ or $f(b)$ is
$(q\exp(C\sqrt{\log q\log\log q}))^{-2}$-badly approximable.
\end{theorem}
}

\subsection{Continued fractions of Laurent series} \label{ss_CF_LS}

Consider the set $\Q((z^{-1}))$ of Laurent series equipped with the standard
valuation which is defined as follows: for $f(z) = \sum_{k=-d}^\infty
c_kz^{-k} \in \Q((z^{-1}))$, its valuation $\|f(z)\|$ is the biggest degree
$d$ of $z$ having non-zero coefficient $c_{-d}$. For example, for polynomials
$f(z)$ the valuation $\|f(z)\|$ coincides with their degree. It is well known
that in this setting the notion of continued fraction is well defined. In
other words, every $f(z)\in \Q((z^{-1}))$ can be written as
$$
f(z) = [a_0(z), a_1(z),a_2(z),\ldots] = a_0(z) + \bigk_{n=1}^\infty
\frac{1}{a_n(z)},
$$
where $a_i(z)$, $i\in\Z_{\ge 0}$, are non-zero polynomials with rational coefficients of degree at least
1.

The continued fractions of Laurent series share most of the properties of classical
ones~\cite{poorten_1998}. Furthermore, in this setting we have even stronger version of Legendre
theorem:
\begin{theoreml}\label{th_legendre}
Let $f(z)\in \Q((z^{-1}))$. Then $p(z)/q(z)\in \Q(z)$ in a reduced form is a
convergent of $f(z)$ if and only if
$$
\left\| f(z) - \frac{p(z)}{q(z)}\right\| < -2\|q(z)\|.
$$
\end{theoreml}
Its proof can be found in~\cite{poorten_1998}. Moreover, if $p_k(z)/q_k(z)$
is the $k$th convergent of $f(z)$ in its reduced form, then
\begin{equation}\label{eq_conv}
\left\| f(z) - \frac{p_k(z)}{q_k(z)}\right\| = %Here, unlike in real case, we have an equality
-\|q_k(z)\|-\|q_{k+1}(z)\|.
\end{equation}

For a Laurent series $f(z)\in \Q((z^{-1}))$, consider its value $f(b)$, where
$b\in \N$ lies within the disc of convergence of $f$. It is well known that
the continued fraction of $f(b)$ (or indeed of any real number $x$) encodes,
in a pretty straightforward way, approximational properties of this number.
At the same time, it is a much subtler question how to read such properties
of $f(b)$ from the continued fraction of $f(z)$. The problem comes from the
fact that after specialization at $z=b$, partial quotients of $f(z)$ become
rational, but often not integer numbers, or they may even vanish. Therefore
the necessary recombination of partial quotients is often needed to construct
the proper continued fraction of $f(b)$. The problem of this type has been
studied in the beautiful article~\cite{vdP_S}. Despite this complication, in
many cases some information on Diophantine approximaiton properties of $f(b)$
can be extracted. In particular, this is the case for Mahler numbers.
Bugeaud, Han, Wen and Yao~\cite{bugeaud_han_wen_yao_2015} provided the
following result that links the continued fraction of $f(z)$ and the
irrationality exponents of values $f(b), b\in \N$. In fact, they formulated
it in terms of Hankel determinants. The present reformulation can be found
in~\cite{badziahin_2017}:

\begin{theorembhwy}
Let $d\ge 2$ be an integer and $f(z) = \sum_{n=0}^\infty c_nz^n$ converge
inside the unit disk. Suppose that there exist integer polynomials $A(z),
B(z), C(z), D(z)$ with $B(0)D(0)\neq 0$ such that
\begin{equation}\label{mahl_eq}
f(z) = \frac{A(z)}{B(z)} + \frac{C(z)}{D(z)}f(z^d).
\end{equation}
Let $b\ge 2$ be an integer such that $B(b^{-d^n})C(b^{-d^n})D(b^{-d^n})\neq
0$ for all $n\in\Z_{\ge 0}$. Define
$$
\rho:= \limsup_{k\to\infty} \frac{\deg q_{k+1}(z)}{\deg q_k(z)},
$$
where $q_k(z)$ is the denominator of $k$th convergent to $z^{-1}f(z^{-1})$.
Then $f(1/b)$ is transcendental and
$$
\mu(f(1/b))\le (1+\rho)\min\{\rho^2,d\}.
$$
\end{theorembhwy}

The corollary of this theorem is that, as soon as
\begin{equation}\label{cond_cf}
\limsup_{k\to\infty} \frac{\deg q_{k+1}(z)}{\deg q_k(z)}=1,
\end{equation}
the irrationality exponent of $f(1/b)$ equals two. Then the natural question
arises: can we say anything better on the Diophantine approximation
properties of $f(1/b)$ in the case if the continued fraction of
$z^{-1}f(z^{-1})$ satisfies a stronger condition than~\eqref{cond_cf}? In
particular, what if the degrees of all partial quotients $a_k(z)$ are bounded
by some absolute constant or even are all linear? Here we answer this
question for a subclass of Mahler
functions. %, notably for the Mahler functions verifying the following Assumption~\ref{assumption_badly_approximable}.
%\begin{assumption}

%\end{assumption}

The main result of this paper is the following.
\begin{theorem}\label{th_main}
Let $d\ge 2$ be an integer and
\begin{equation}\label{def_f}
f(z) = \prod_{t=0}^\infty P(z^{-d^t}),
\end{equation}
where $P(z)\in \Z[z]$ is a polynomial such that $P(1) = 1$ and $\deg P(z)<d$.
Assume that the series $f(z)$ is badly approximable (i.e. the degrees of all
partial quotients of $f(z)$ are bounded from above by an absolute constant).
Then there exists a positive constant $K$ such that for any $b\in\Z$,
$|b|\geq 2$, we have either $f(b) = 0$ or $f(b)$ is $q^{-2}\exp(-K\sqrt{\log
q\log\log q})$-badly approximable.
\end{theorem}

\section{Preliminary information on series $f(z)$.}

In the further discussion, we consider series $f(z)$ which satisfies all the
conditions of Theorem~\ref{th_main}. Most of these conditions are straightforward to verify, the only non-evident point is to check whether the product function $f(z)$, defined by~\eqref{def_f}, is badly approximable. To address this, one can find a nice criteria %for $f(z)$ to be badly approximable
 in \cite[Proposition 1]{badziahin_2017}: $f(z)$ is badly
approximable if and only if all its partial quotients are linear. This in
turn is equivalent to the claim that the degree of denominator of the $k$th convergent of $f(z)$
is precisely $k$, for all $k\in\N$.

As shown in~\cite{badziahin_2017}, it is easier to compute the continued
fraction of a slightly modified series
\begin{equation} \label{def_g}
g(z) = z^{-1}f(z).
\end{equation}
Since Diophantine approximtion properties of numbers $f(b)$ and $g(b) =
f(b)/b$ essentially coincide, for any $b\in \N$, we will further focus on the
work with the function $g(z)$. As we assume that $f(z)$ is a badly
approximable function, the function $g(z)$ defined by~\eqref{def_g} is also
badly approximable. In what follows, we will denote by $p_k(z)/q_k(z)$ the
$k$th convergent of $g(z)$, and then, by~\cite[Proposition 1]{badziahin_2017},
we infer that $\deg q_k(z)=k$.

Write down the polynomial $P(z)$ in the form
$$
P(z) = 1 + u_1z+\ldots +u_{d-1}z^{d-1}.
$$
Then $P(z)$ is defined by the vector $\vu = (u_1,\ldots,u_{d-1}) \in
\Z^{d-1}$ and, via~\eqref{def_f} and~\eqref{def_g}, so is $g(z)$. To
emphasize this fact, we will often write $g(z)$ as $g_\vu(z)$.

\subsection{Coefficients of the series, convergents and Hankel determinants}

We write the Laurent series $g_\vu(z)\in \Z[[z^{-1}]]$ in the following form
\begin{equation} \label{def_g_u}
g_\vu(z) = \sum_{n=1}^\infty c_n z^{-n}.
\end{equation}
We denote by $\vc_n$ the vector $(c_1,c_2,\ldots, c_n)$.
%, where the
%coefficients $c_n$, $n\in\N$, are given by~\eqref{def_g_u}.
Naturally, the definition of $g_{\vu}(z)$ via the infinite product
(see~\eqref{def_f} and~\eqref{def_g}) imposes the upper bound on  $|c_n|$,
$n\in\N$.% Such a bound is proven in Lemma~\ref{lemma_ub_c_i}.
\begin{lemma} \label{lemma_ub_c_i}
The term $c_n$ satisfies
\begin{equation}\label{ineq_c_n}
|c_n| \leq \|\vu\|_\infty ^{\lceil \log_d n\rceil} \le \|\vu\|_\infty ^{\log_d n + 1}.
\end{equation}
Consequently,
\begin{equation}\label{ineq_c}
\|\vc_n\|_\infty \leq \|\vu\|_\infty ^{\log_d n + 1}
\end{equation}
\end{lemma}
\begin{proof}
Look at two different formulae for $g_\vu(z)$:
$$
g_\vu(z) = z^{-1}\prod_{t=0}^\infty (1+u_1z^{-d^t}+\ldots+u_{d-1}z^{-(d-1)d^t}) = \sum_{n=1}^\infty c_n z^{-n}.
$$
By comparing the right and the left hand sides one can notice that $c_n$ can be
computed as follows:
\begin{equation} \label{link_c_u}
c_n = \prod_{j=0}^{l(n)} u_{d_{n,j}}
\end{equation}
where $d_{n,0}d_{n,1}\cdots d_{n,l(n)}$ is the $d$-ary expansion of the number
$n-1$. Here we formally define $u_0=1$. Equation~\eqref{link_c_u} readily implies that $|c_n|\leq
\|\vu\|^{l(n)}$. Finally, $l(n)$ is estimated by $l(n) \leq \lceil
\log_d(n-1)\rceil \le \lceil \log_d n\rceil$. The last two inequalities
clearly imply~\eqref{ineq_c_n}, hence~\eqref{ineq_c}.
\end{proof}

Let $p_k(z)/q_k(z)$ be a convergent of $g_\vu(z)$ in its reduced form. Recall that throughout the text we assume that $f(z)$ is badly approximable, hence $g_{\vu}(z)$ defined by~\eqref{def_g} is badly approximable, and because of this (and employing~\cite[Proposition 1]{badziahin_2017}) we have
\begin{equation} \label{deg_qk_k}
\deg q_k=k.
\end{equation}
Denote
\begin{equation} \label{coefficients_q_k_p_k}
\begin{aligned}
q_k(z) &= a_{k,0}+a_{k,1}z+\ldots + a_{k,k}z^k,\qquad \va_k:=(a_{k,0},\ldots,a_{k,k})\\
p_k(z) &= b_{k,0} +
b_{k,1}z+\ldots+b_{k,k-1}z^{k-1},\quad \vb_k:=(b_{k,0},\ldots,b_{k,k-1}).
\end{aligned}
\end{equation}
Because of~\eqref{deg_qk_k}, we have
\begin{equation} \label{a_k_b_k_ne_0}
%\begin{aligned}
a_{k,k}\ne 0.%,\\
%b_{k,k-1}\ne 0.
%\end{aligned}
\end{equation}

The Hankel matrix is defined as follows:
$$
H_k=H_k(g_{\vu})=
\begin{pmatrix}
c_{1} & c_{2} & \dots & c_{k}\\
c_{2} & c_{3} & \dots & c_{k+1}\\
\vdots & \vdots & \ddots & \vdots\\
c_{k} & c_{k+l} & \dots & c_{2k-1}
\end{pmatrix}.
$$
It is known (see, for example,~\cite[Section 3]{badziahin_2017}) that the
convergent in its reduced form with $\deg q_k(z) = k$ exists if and only if the
Hankel matrix $H_k$ is invertible. Thus in our case
%, because we have~\eqref{deg_qk_k},
we necessarily have that $H_k(g_{\vu})$
is invertible for any positive integer $k$.

From~\eqref{eq_conv}, we have that
\begin{equation}\label{eq_pq}
\|q_k(z)g_\vu(z) - p_k(z)\| = -k-1.
\end{equation}
In other words, the coefficients for $x^{-1}, \ldots, x^{-k}$ in
$q_k(z)g_\vu(z)$ are all zero and the coefficient for $x^{-k-1}$ is not. This
suggests a method for computing $q_k(x)$. One can check that the vector
$\va_k = (a_{k,0},a_{k,1},\ldots, a_{k,k})$ is the solution of the matrix
equation $H_{k+1}\va_k = c\cdot\ve_{k+1}$, where $c$ is a non-zero
constant and
$$
\ve_{k+1}=(0, \dots, 0, 1)^{t}.
$$
This equation has the unique solution since the matrix $H_{k+1}$ is
invertible. So, we can write the solution vector $\va$ as
\begin{equation} \label{matrix_equality_HAI}
\va_k = c\cdot H_{k+1}^{-1}\ve_{k+1}.
\end{equation}

In what follows, we will use the norm of the matrix $\|H\|_\infty$, defined
to be the maximum of the absolute values of all its entries. Given a
polynomial $P(z)$ we define its height $h(P)$ as the maximum of absolute
values of its coefficients. In particular, we have $h(p_k(z)) =
\|\vb_k\|_\infty$ and $h(q_k(z)) = \|\va_k\|_\infty$.
%Finally, by $\vc_k$ we denote the vector $(c_1,c_2,\ldots, c_k)$, where the coefficients $c_i$, $i\in\N$, are given by~\eqref{def_g_u}.

\hidden{
\begin{lemma} \label{lemma_ub_c_i}
The term $c_i$ satisfies
$$
|c_i| \le \|\vu\|_\infty ^{\lceil \log_d i\rceil} \le \|\vu\|_\infty ^{\log_d i + 1}.
$$
This in fact implies
\begin{equation}\label{ineq_c}
\|\vc_i\|_\infty \le \|\vu\|_\infty ^{\log_d i + 1}
\end{equation}
\end{lemma}
\begin{proof}
Look at two different formulae for $g_\vu(z)$:
$$
g_\vu(z) = z^{-1}\prod_{t=0}^\infty (1+u_1z^{-d^t}+\ldots+u_{d-1}z^{-(d-1)d^t}) = \sum_{i=1}^\infty c_i z^{-i}.
$$
By comparing the right and left hand sides one can notice that $c_i$ can be
computed as follows:
$$
c_i = \prod_{j=0}^{l(i)} u_{d_{i,j}}
$$
where $d_{i,0}d_{i,1}\cdots d_{i,l(i)}$ is a $d$-ary expansion of the number
$i-1$. Here we put $u_0=1$. This formula readily implies that $|c_i|\le
\|\vu\|^{l(i)}$. Finally, $l(i)$ is estimated by $l(i) \le \lceil
\log_d(i-1)\rceil \le \lceil \log_d i\rceil$. The last two inequalities
imply~\eqref{ineq_c}.
\end{proof}
}

\begin{lemma}\label{lem_hpq}
For any $k\in\N$, the $k$-th convergent $p_k(z)/q_k(z)$ to $g_{\vu}(z)$ can be represented by $p_k(z)/q_k(z)=\tilde{p}_k(z)/\tilde{q}_k(z)$, %of $g_\vu(z)$ such that
where $\tilde{p}_k,\tilde{q}_k\in \Z[z]$ and
\begin{eqnarray}\label{ineq_hq}
h(\tilde{q}_k)&\leq& (\|\vc_{2k+1}\|_\infty^2 \cdot k)^{k/2}.\\
\label{ineq_hp}
h(\tilde{p}_k)&\leq& \|\vc_{2k+1}\|_\infty^{k+1} \cdot k^{(k+2)/2}.
\end{eqnarray}
Consecutively, the following upper bounds hold true:
\begin{eqnarray}\label{ineq_hq_2}
h(\tilde{q}_k)&\leq& \|\vu\|_\infty ^{k(\log_d (2k+1) + 1)} \cdot k^{k/2}.\\
\label{ineq_hp_2}
h(\tilde{p}_k)&\leq& \|\vu\|_\infty ^{(k+1)(\log_d (2k+1) + 1)} \cdot k^{(k+2)/2}.
\end{eqnarray}
\end{lemma}

\begin{proof} By applying Cramer's rule to the equation $H_{k+1}\va_k = c\cdot\ve_{k+1}$ we infer that
\begin{equation} \label{formula_a_i}
a_{k,i}=c\cdot\frac{\Delta_{k+1,i}}{\det H_{k+1}}, \quad i=0,\dots,k,
\end{equation}
where $\Delta_{k+1,i}$ denotes the determinant of the matrix $H_{k+1}$ with
the $i$-th column replaced by $\ve_{k+1}$, $i=1,\dots,k+1$. Then we use the
Hadamard's determinant upper bound to derive
\begin{equation} \label{h_a_i}
|\det H_{k+1}|\leq \|H_{k+1}\|_\infty^{k+1}\cdot (k+1)^{(k+1)/2} = (\|\vc_{2k+1}\|_\infty^2 (k+1))^{(k+1)/2}.
\end{equation}
Moreover, by expanding the matrix involved in $\Delta_{k+1,i}$ along the $i$th
column and by using Hadamard's upper bound again we get
$$
|\Delta_{k+1,i}|\le \|H_{k+1}\|_\infty^k\cdot  k^{k/2} = (\|\vc_{2k+1}\|_\infty^2\cdot  k)^{k/2}, \quad i=0,\dots,k.
$$

To define $\tilde{q}_k(z)$, set  $c = \det H_{k+1}$ in~\eqref{formula_a_i}.
Then we readily have  $\tilde{q}_k(z) = \sum_{i=0}^k \Delta_{k+1,i+1}z^i$. By
construction, it has integer coefficients and $h(\tilde{q})$
satisfies~\eqref{ineq_hq}.

Next, from~\eqref{eq_pq} we get that the coefficients of $\tilde{p}_k(z)$ coincide with
the coefficients for positive powers of $z$ of $\tilde{q}_k(z)g_\vu(z)$. By expanding
the latter product, we get
$$
|b_{k,i}|=\left|\sum_{j=i+1}^k a_{k,j} c_{j-i-1}\right| \le \|\vc_{2k+1}\|_{\infty}^{k+1}\cdot k^{(k+2)/2}.
$$
%Then for $\tilde{p}_k(z) = \det H_{k+1} p_k(z)$ the $i$th coefficient is
%an integer and satisfies
%$$
%|\det H_{k+1} b_i| = \left|\sum_{j=i+1}^k b_{k+1,j+1}c_{j-i}\right| \le \|\vc_{2k+1}\|_{\infty}^{k+1}\cdot k^{(k+2)/2}.
%$$
Hence~\eqref{ineq_hp} is also satisfied.

The upper bounds~\eqref{ineq_hq_2} and~\eqref{ineq_hp_2}  follow
from~\eqref{ineq_hq} and~\eqref{ineq_hp} respectively by applying
Lemma~\ref{lemma_ub_c_i}.
\end{proof}

\begin{notation}
For the sake of convenience, further in this text we will assume that all the convergents to $g_{\vu}(z)$ are in the form described in Lemma~\ref{lem_hpq}. That is, %we will write $p_k(z)$ and $q_k(z)$ instead of $\tilde{p}_k(z)$ and $\tilde{q}_k(z)$, in particular
we will always assume that $p_k(z)$ and $q_k(z)$ have integer coefficients
and verify the upper bounds~\eqref{ineq_hq} and~\eqref{ineq_hp}, as well
as~\eqref{ineq_hq_2} and~\eqref{ineq_hp_2}.
\end{notation}

For any $k\in\N$ we define a suite of coefficients $(\alpha_{k,i})_{i> k}$ by
\begin{equation} \label{def_alpha}
q_k(z)g_\vu(z)-p_k(z)=:\sum_{i=k+1}^{\infty}\alpha_{k,i} z^{-i}.
\end{equation}

Note that from the equation $H_{k+1}\va_k = c\cdot\ve_{k+1}$ we can get that
$\alpha_{k,k+1} = c = \det H_{k+1}$. In particular, it is a non-zero integer.
\begin{lemma} \label{lemma_ub_coefficients}
For any $i,k\in\N$, $i> k\geq 1$, we have
\begin{equation} \label{lemma_ub_coefficients_claim}
\begin{aligned}
|\alpha_{k,i}|&\leq (k+1) \|\vc_{k+i}\|_\infty (\|\vc_{2k+1}\|_\infty^2 \cdot k)^{k/2}\\
&\leq (k+1) \|\vu\|_\infty^{\log_d(k+i)+1} \|\vu\|_\infty^{k(\log_d(2k+1)+1)} \cdot k^{k/2}.
\end{aligned}
\end{equation}
\end{lemma}

\begin{proof}
One can check that $\alpha_{k,i}$ is defined by the formula $\alpha_{k,i} =
\sum_{j=0}^k a_{k,j} c_{j+i}$, which in view of~\eqref{ineq_hq} from
Lemma~\ref{lem_hpq} implies the first inequality
in~\eqref{lemma_ub_coefficients_claim}. Then, the second inequality
in~\eqref{lemma_ub_coefficients_claim} follows by applying
Lemma~\ref{lemma_ub_c_i}.
\end{proof}

\hidden{
\subsection{Estimates on the terms of the series}

\begin{lemma}
The term $c_i$ satisfies
$$
|c_i| \le \|\vu\|_\infty ^{\lceil \log_d i\rceil} \le \|\vu\|_\infty ^{\log_d i + 1}.
$$
This in fact implies
\begin{equation}\label{ineq_c}
\|\vc_i\|_\infty \le \|\vu\|_\infty ^{\log_d i + 1}
\end{equation}
\end{lemma}
\proof

Look at two different formulae for $g_\vu(z)$:
$$
g_\vu(z) = z^{-1}\prod_{t=0}^\infty (1+u_1z^{-d^t}+\ldots+u_{d-1}z^{-(d-1)d^t}) = \sum_{i=1}^\infty c_i z^{-i}.
$$
By comparing the right and left hand sides one can notice that $c_i$ can be
computed as follows:
$$
c_i = \prod_{j=0}^{l(i)} u_{d_{i,j}}
$$
where $d_{i,0}d_{i,1}\cdots d_{i,l(i)}$ is a $d$-ary expansion of the number
$i-1$. Here we put $u_0=1$. This formula readily implies that $|c_i|\le
\|\vu\|^{l(i)}$. Finally, $l(i)$ is estimated by $l(i) \le \lceil
\log_d(i-1)\rceil \le \lceil \log_d i\rceil$. The last two inequalities
imply~\eqref{ineq_c}.
\endproof
}

\subsection{Using functional equation to study Diophantine approximaiton properties.}

From~\eqref{def_f} one can easily get a functional equation for $g_\vu(z) =
z^{-1}f(z)$:
\begin{equation}\label{eq_funcg}
g_\vu (z) = \frac{g_\vu (z^d)}{P^*(z)},\quad P^*(z) = z^{d-1}P(z^{-1}) = z^{d-1}+u_1z^{d-2}+\ldots+u_{d-1}.
\end{equation}
This equation allows us, starting from the convergent $p_k(z)/q_k(z)$ to
$g_\vu(z)$, to construct an infinite sequence of convergents
$\left(p_{k,m}(z)/q_{k,m}(z)\right)_{m\in\N_0}$ to $g_\vu(z)$ by
\begin{equation} \label{def_pkm_qkm_poly}
\begin{aligned}
q_{k,m}(z)&:=q_k(z^{d^m}),\\
p_{k,m}(z)&:=\prod_{t=0}^{m-1}P^*(z^{d^t})p_{k}(z^{d^m}).
\end{aligned}
\end{equation}
This fact can be checked by substituting %recursively $z^d$ in place of $z$ to
the functional equation~\eqref{eq_funcg} into the condition of Theorem~L. The reader can
also compare with~\cite[Lemma 3]{badziahin_2017}.

By employing~\eqref{eq_funcg} and~\eqref{def_alpha}, we find
\begin{equation} \label{property_k_m}
q_{k,m}(z)g_\vu(z)-p_{k,m}(z)=\prod_{t=0}^{m-1}P^*(z^{d^t})\cdot\sum_{i=k+1}^{\infty}\alpha_{k,i}
z^{-d^m\cdot i}.
\end{equation}

Consider an integer value $b$ which satisfies the conditions of
Theorem~\ref{th_main}. Define\footnote{There is a slight abuse of notation in
using the same letters $p_{k,m}$ and $q_{k,m}$ both for polynomials from
$\Z[z]$ and for their values at $z=b$. However, we believe that in this
particular case such a notation constitutes the best choice. Indeed, the main
reason to consider polynomials $p_{k,m}(z)$ and $q_{k,m}(z)$ is to define
eventually $p_{k,m}=p_{k,m}(b)$ and $q_{k,m}=q_{k,m}(b)$, which will play the
key role in the further proofs. At the same time, it is easy to distinguish
the polynomials $p_{k,m}(z)$, $q_{k,m}(z)$ and the corresponding integers
$p_{k,m}$ and $q_{k,m}$ by the context. Moreover, we will always specify
which object we mean and always refer to the polynomials specifying
explicitly the variable, that is $p_{k,m}(z)$, $q_{k,m}(z)$ and not $p_{k,m}$
and $q_{k,m}$.}
\begin{eqnarray}
p_{k,m}&:=&p_{k,m}(b), \label{def_p_k_m}\\
q_{k,m}&:=&q_{k,m}(b), \label{def_q_k_m}
\end{eqnarray}
where $p_{k,m}(z)$ and $q_{k,m}(z)$ are polynomials defined by~\eqref{def_pkm_qkm_poly}.

Clearly, for any $k\in\N$, $m\in\N_0$ we have $p_{k,m},q_{k,m}\in\Z$.

\hidden{
\begin{remark}
There is a slight abuse of notation in using the same letters $p_{k,m}$ and
$q_{k,m}$ both for polynomials from $\Z[z]$ and for their values at $z=b$.
However, we believe that in this particular case such a notation constitutes
the best choice. Indeed, the main reason to consider polynomials $p_{k,m}(z)$
and $q_{k,m}(z)$ is to define eventually $p_{k,m}=p_{k,m}(b)$ and
$q_{k,m}=q_{k,m}(b)$, which will play the key role in the further proofs. At
the same time, it is easy to distinguish the polynomials $p_{k,m}(z)$,
$q_{k,m}(z)$ and the corresponding integers $p_{k,m}$ and $q_{k,m}$ by the
context. Moreover, we will always specify which object we mean and always
refer to the polynomials specifying explicitly the variable, that is
$p_{k,m}(z)$, $q_{k,m}(z)$ and not $p_{k,m}$ and $q_{k,m}$.
%(the latter
%notation is reserved in the further text, of course, for the corresponding
%values $p_{k,m}(b)$ and $q_{k,m}(b)$).
\end{remark}
}

\begin{lemma} \label{lemma_smallness}
Let $b,k,m\in\N$, $b\geq 2$. Assume
\begin{equation} \label{lemma_smallness_bdm_geq_2}
b^{d^m}>2^{1+\log_d\|\vu\|_\infty}.
\end{equation}
Then the integers $p_{k,m}$ and $q_{k,m}$ verify
\begin{equation} \label{lemma_smallness_ub}
\left|g_\vu(b)-\frac{p_{k,m}}{q_{k,m}}\right|\leq \frac{2(k+1)
k^{k/2} d^m \|\vu\|_\infty^{m+(k+1)(\log_d(2k+1)+1)}}{q_{k,m}\cdot
b^{d^m\cdot k+1}}.
\end{equation}
Moreover, if we make in addition a stronger assumption
\begin{equation}\label{lem5_assump}
b^{d^m}\ge 4(k+1)k^{k/2} \|\vu\|_\infty^{(k+1)(\log_d(2k+1)+1)},
\end{equation}
then
\begin{equation} \label{lemma_smallness_lb}
\frac{|g_\vu(b)|}{4 q_{k,m}\cdot b^{d^m\cdot
k+1}}\leq\left|g_\vu(b)-\frac{p_{k,m}}{q_{k,m}}\right|.
\end{equation}
\end{lemma}

\begin{proof} Consider Equation~\eqref{property_k_m} with substituted $z:=b$:
\begin{equation} \label{property_k_m_z_substituted_b}
q_{k,m}g_\vu(b)-p_{k,m}=\prod_{t=0}^{m-1}P^*(b^{d^t})\cdot\sum_{i=k+1}^{\infty}\alpha_{k,i}
b^{-d^m\cdot i}.
\end{equation}
Each of the factors in $\left|P^*(b^{d^t})\right|$ in the
right hand side of~\eqref{property_k_m_z_substituted_b} can be upper bounded by
$d\cdot\|\vu\|_\infty b^{d^{t}(d-1)}$.
So, the product in the
right hand side of~\eqref{property_k_m_z_substituted_b} can be estimated by
\begin{equation}\label{eq_lem5_1}
\left|\prod_{t=0}^{m-1}P^*(b^{d^t})\right| \leq d^m \|\vu\|_\infty^m \cdot
b^{d^{m}-1}.
\end{equation}
Further, we estimate the second term on the right hand
side of~\eqref{property_k_m_z_substituted_b} by employing Lemma~\ref{lemma_ub_coefficients}:
\begin{equation}\label{eq_lem5_2}
\left|\sum_{i=k+1}^{\infty}\alpha_{k,i} b^{-d^m\cdot i}\right| \leq
%(k+1)(\|\vc_{2k+1}||^2_\infty\cdot k)^{k/2} \sum_{i=k+1}^\infty
%\frac{\|\vc_{k+i}\|_\infty}{b^{d^m\cdot i}}.
(k+1) \|\vu\|_\infty^{k(\log_d(2k+1)+1)} \cdot k^{k/2}\sum_{i=k+1}^\infty
\frac{\|\vu\|_\infty^{\log_d(k+i)+1}}{b^{d^m\cdot i}}.
\end{equation}
The last sum in the right hand side of~\eqref{eq_lem5_2} is bounded from
above by
$$
\sum_{i=k+1}^\infty
\frac{\|\vu\|_\infty^{\log_d(k+i)+1}}{b^{d^m\cdot i}} \le
\frac{\|\vu\|_\infty}{b^{d^m(k+1)}}\cdot \sum_{i=0}^\infty
\frac{\|\vu\|_\infty^{\log_d(2k+1+i)}}{b^{d^m\cdot i}}
$$
\begin{equation}\label{eq_lem5_3}
\le \frac{\|\vu\|_\infty^{1+\log_d(2k+1)}}{b^{d^m(k+1)}}
\sum_{i=0}^\infty \frac{(i+1)^{\log_d\|\vu\|_\infty}}{b^{d^m\cdot
i}} \le \frac{\|\vu\|_\infty^{1+\log_d(2k+1)} \cdot
C(b,d,m,\|\vu\|_\infty)}{b^{d^m(k+1)}},
\end{equation}
where
\hidden{
$C(\|\vu\|_\infty)$ is a constant which only depends on
$\|\vu\|_\infty$ but not on $k$ or $m$. In particular, for
$\|\vu\|_\infty = 1$ it can be made equal to 2. It is also not too
difficult to check that $C(\|\vu\|_\infty)\le 2$ in the case
$b^{d^m}\ge 2^{1+\log_d \|\vu\|_\infty}$. Indeed, it readily follows from elementary remark that $i+1\leq 2^i$ for all $i\in\Z$.
}
$$
C(b,d,m,\|\vu\|_\infty)=\sum_{i=0}^\infty \frac{(i+1)^{\log_d\|\vu\|_\infty}}{b^{d^m\cdot
i}}.
$$
Note that for any $i\in\Z$, we have $i+1\leq 2^i$. Because of this, assumption~\eqref{lemma_smallness_bdm_geq_2} implies
\begin{equation} \label{C_leq_2}
C(b,d,m,\|\vu\|_\infty)\leq 2.
\end{equation}

Finally, by putting
together,~\eqref{property_k_m_z_substituted_b},~\eqref{eq_lem5_1},~\eqref{eq_lem5_2},~\eqref{eq_lem5_3}
and~\eqref{C_leq_2} we get
$$
\left|q_{k,m}g_\vu(b) - p_{k,m}\right|\le \frac{2(k+1) k^{k/2} d^m
\|\vu\|_\infty^{m+(k+1)(\log_d(2k+1)+1)}}{b^{d^m\cdot k + 1}}.
$$
Dividing both sides by $q_{k,m}$ gives~\eqref{lemma_smallness_ub}.

To get the lower bound, we first estimate the product
in~\eqref{property_k_m}.
$$
\prod_{t=0}^{m-1}P^*(b^{d^t}) = b^{d^m-1}
\prod_{t=0}^{m-1}P(b^{-d^t}) \geq
b^{d^m-1}\frac{g_\vu(b)}{\prod_{t=m}^{\infty}P(b^{-d^t})}.
$$
By~\eqref{lem5_assump}, the denominator can easily be estimated as
$$
\prod_{t=m}^{\infty}P(b^{-d^t})\le \prod_{t=m}^{\infty} \left(1+\frac{2\|\vu\|_\infty}{b^{d^t}}\right) < 2.
$$
Therefore,
$$
\prod_{t=0}^{m-1}P^*(b^{d^t})\ge \frac12b^{{d^m}-1}g_\vu(b).
$$

For the series in the right hand side of~\eqref{property_k_m}, we show that
the first term dominates this series. Indeed, we have $|\alpha_{k,k+1}|\ge 1$
since it is a non-zero integer. Then,
\begin{equation} \label{lb_calculations}
\begin{aligned}
&|q_{k,m}g_\vu(b)-p_{k,m}|=\left|\prod_{t=0}^{m-1}P^*(b^{d^t})\cdot\sum_{i=k+1}^{\infty}\alpha_{k,i} b^{-d^m\cdot i}\right|\\
&\geq
\frac12b^{d^m-1} |g_\vu(b)| \left(b^{-d^m(k+1)}-\sum_{i=k+2}^{\infty}\left|\alpha_{k,i}\right| b^{-d^m\cdot i}\right)\\
&\stackrel{\eqref{eq_lem5_2}, \eqref{eq_lem5_3}}\geq \frac12b^{d^m-1}
|g_\vu(b)| \left(b^{-d^m(k+1)} -
\frac{C(b,d,m,\|\vu\|_\infty)(k+1)k^{k/2}\|\vu\|_\infty^{(k+1)(\log_d(2k+1)+1)}}{b^{d^m(k+2)}}\right)
\end{aligned}
\end{equation}
Recall that %we have already verified
by~\eqref{C_leq_2}, we have
$C(b,d,m,\|\vu\|_\infty) \leq 2$. So, by using
assumption~\eqref{lem5_assump}, we finally get
$$
|q_{k,m}g_\vu(b)-p_{k,m}|\ge
\frac{b^{d^m-1}|g_\vu(b)|}{4b^{d^m(k+1)}} =
\frac{|g_\vu(b)|}{4b^{d^m\cdot k +1}}.
$$
Finally, dividing both sides by $q_{k,m}$ leads
to~\eqref{lemma_smallness_lb}.
\end{proof}

\begin{lemma} \label{lemma_qkm_db}
Let $b,k,m\in\N$, $k\geq 1$ and let
\begin{equation}  \label{lemma_qkm_db_bdm_assumption}
b^{d^m} > 3\cdot
\left(\|\vc_{2k+1}\|_\infty^2 k\right)^{k/2}.
\end{equation}
Recall the notations $a_{k,i}$, $i=0,\dots,k$, %and $b_{k,i}$, $i=0,\dots,k-1$,
for the coefficients of $q_k$, %and $p_k$ respectively,
$k\in\N$, is defined in~\eqref{coefficients_q_k_p_k}. Then,
\begin{equation} \label{lemma_qkm_db_claim}
\frac12|a_{k,k}|\cdot b^{kd^m}\leq
q_{k,m}\leq\frac32|a_{k,k}|\cdot b^{kd^m}.
\end{equation}
\end{lemma}
\begin{proof}
The leading term of $q_{k,m}(z)$ is $a_{k,k}z^{kd^m}$. We know that $\deg
q_k(z) = k$, therefore $a_{k,k}\neq 0$ and $a_{k,k}$ is an integer. Recall
also that by~\eqref{ineq_hq} the maximum of the coefficients $a_{k,i}$,
$i=0,\dots,k$, does not exceed $(\|\vc_{2k+1}\|_\infty^2 \cdot k)^{k/2}$.
Thus we find, by using assumption~\eqref{lemma_qkm_db_bdm_assumption},
$$
\left|\sum_{n=0}^{k-1} a_{k,n}\cdot b^{nd^m}\right| \le b^{kd^m}
\left|\sum_{n=1}^{k} 3^{-n}\right| \le \frac12 b^{kd^m}.
$$
We readily infer, by taking into account $q_{k,m} = a_{k,0}+a_{k,1}b^{d^m}+\ldots + a_{k,k}b^{k d^m}$,
$$
\frac12 |a_{k,k}|b^{kd^m}\le |a_{k,k}|b^{kd^m}- \frac12
b^{kd^m}\le |q_{k,m}|\le |a_{k,k}|q^{kd^m}+ \frac12 q^{kd^m} =
\frac32|a_{k,k}|q^{kd^m}.
$$
This completes the proof of the lemma.
\end{proof}

\begin{proposition} \label{proposition_db}
Let $k\geq 2$, $m\geq 1$ be integers and assume that~\eqref{lem5_assump} is
satisfied. Then, the integers $p_{k,m}=p_{k,m}(b)$ and $q_{k,m}=q_{k,m}(b)$,
defined by~\eqref{def_p_k_m} and by~\eqref{def_q_k_m}, satisfy
\begin{equation} \label{proposition_db_l_1}
\left|g_\vu(b)-\frac{p_{k,m}}{q_{k,m}}\right|\leq \frac{3(k+1)
k^{k} d^m \|\vu\|_\infty^{m+(2k+1)(\log_d(2k+1)+1)}}{b \cdot q^2_{k,m}},
\end{equation}
\begin{equation} \label{proposition_db_u_1}
\frac{|g_\vu(b)|}{8b q_{k,m}^2} \le \left|g_\vu(b) -
\frac{p_{k,m}}{q_{k,m}}\right|.
\end{equation}
Moreover, if, additionally, $k$ and $m$ satisfy
\begin{equation} \label{proposition_db_assump}
k\cdot d^m\logc b - 1 \geq
\frac{1}{3} m^2(\log \|\vu\|_\infty)^2,
\end{equation}
then
\begin{equation} \label{proposition_db_claim_2}
\left|g_\vu(b)-\frac{p_{k,m}}{q_{k,m}}\right|\leq
\frac{3\cdot 2^{C{\sqrt{\logc q_{k,m}\logc\logc q_{k,m}}}}}{q_{k,m}^{2}},
\end{equation}
where
\begin{equation} \label{def_C}
C=2\sqrt{2}+2\sqrt{5\cdot\log\|\vu\|_{\infty}}+2.
\end{equation}
%$C = \max\{\sqrt{18}, 6\sqrt{\log||\vu||_\infty}\}$.
\end{proposition}
\begin{proof}
From Lemma~\ref{lemma_qkm_db} we have
$$
b^{k\cdot d^m}\ge
\frac{2q_{k,m}}{3|a_{k,k}|}\stackrel{\eqref{ineq_hq_2}}\ge \frac{2q_{k,m}}{3\|\vu\|_\infty^{k(\log_d(2k+1)+1)}
k^{k/2}}.
$$
Similarly, by using $|a_{k,k}|\geq 1$ together with Lemma~\ref{lemma_qkm_db}, we get the lower bound
\begin{equation} \label{proposition_db_simple_ub}
b^{k\cdot d^m}\le 2q_{k,m}.
\end{equation}
These two bounds on $b^{kd^m}$ allow to infer the inequalities~\eqref{proposition_db_l_1} and~\eqref{proposition_db_u_1} straightforwardly from the corresponding bounds in Lemma~\ref{lemma_smallness}.

%From Lemma~\ref{lemma_qkm_db} we know that $q_{k,m}\ge \frac12
%b^{k\cdot d^m}$. Therefore,

We proceed with the proof of the estimate~\eqref{proposition_db_claim_2}. We are going to deduce it as a corollary of~\eqref{proposition_db_l_1}. To this end, we are going to prove, under the assumptions of this proposition,
\begin{equation} \label{eq_prop_first}
(k+1)
k^{k} d^m \|\vu\|_\infty^{m+(2k+1)(\log_d(2k+1)+1)}\leq 2^{C{\sqrt{\logc q_{k,m}\logc\logc q_{k,m}}}},
\end{equation}
where the constant $C$ is defined by~\eqref{def_C}. It is easy to verify
that~\eqref{proposition_db_l_1} and~\eqref{eq_prop_first} indeed
imply~\eqref{proposition_db_claim_2}. Therefore in the remaining part of the
proof we will focus on verifying~\eqref{eq_prop_first}.

The inequality~\eqref{proposition_db_simple_ub} together with condition~\eqref{lem5_assump} imply
\begin{equation}\label{ineq_logqkm}
\begin{array}{rl}
\logc q_{k,m} \geq &(2k-1)+k\logc (k+1) + \frac{k^2}{2}\logc k\\
&+k(k+1)(\log_d(2k+1)+1)\logc\|\vu\|_\infty.
\end{array}
\end{equation}
By taking logarithms again one can derive that $\logc\logc q_{k,m}\ge
\logc k$. Now we compute
\begin{equation} \label{proposition_db_implication_1}
\logc q_{k,m}\logc\logc q_{k,m} \geq \frac{k^2}{2}(\logc k)^2 >
\frac{1}{8}(k\logc k+\logc(k+1))^2.
\end{equation}
The last inequality in~\eqref{proposition_db_implication_1} holds true because $k\logc k>\logc (k+1)$ for $k\geq 2$.

Another implication of~\eqref{ineq_logqkm} is
\begin{equation} \label{proposition_db_implication_2_0}
\logc q_{k,m}\logc\logc q_{k,m} \geq k(k+1)(\log_d(2k+1)+1)\logc k
\logc\|\vu\|_\infty.
\end{equation}
Since for $d\ge 2$ and $k\ge 2$ we have $\logc k\geq
\frac14(\log_d(2k+1)+1)$ and $k(k+1)\geq \frac15(2k+1)^2$, therefore
we readily infer from~\eqref{proposition_db_implication_2_0}
\begin{equation} \label{proposition_db_implication_2}
\logc q_{k,m}\logc\logc q_{k,m} \ge \frac{1}{20\logc\|\vu\|_\infty}
(2k+1)^2(\log_d(2k+1)+1)^2 (\logc \|\vu\|_\infty)^2.
\end{equation}

Next, it follows from~\eqref{proposition_db_simple_ub} that
\begin{equation} \label{qkm_geq_kdm}
\logc q_{k,m} \geq k\cdot d^m\logc b - 1.
\end{equation}
Therefore %$\logc\logc q_{k,m}\geq m$. At the same time, \eqref{proposition_db_implication_2_0} clearly implies $\logc q_{k,m}\geq \logc \|\vu\|_\infty$, and we readily deduce
assumption~\eqref{proposition_db_assump} implies that $\logc q_{k,m}\geq
\frac{1}{3} m^2(\logc \|\vu\|_\infty)^2$. At the same time, the assumptions
$k\geq 2$ joint with~\eqref{lem5_assump} readily imply $b^{k\cdot d^m}\geq
576$, hence, by adding~\eqref{proposition_db_simple_ub}, we find $\logc\logc
q_{k,m}\geq\logc\logc 288>3$. So,
\begin{equation} \label{proposition_db_implication_3}
\logc q_{k,m}\logc\logc q_{k,m} > m^2(\logc \|\vu\|_\infty)^2.
\end{equation}
Also, by these considerations we deduce from~\eqref{qkm_geq_kdm}
\begin{equation} \label{proposition_db_implication_4}
\logc q_{k,m}\logc\logc q_{k,m} > 3 d^m > \left(m\cdot\logc d\right)^2.
\end{equation}
%where the last inequality in~\eqref{proposition_db_implication_4} is justified by $d^m\geq 2$ and so necessarily $2 d^m>\left(m\cdot\logc d\right)^2$
\hidden{
Then,
$$
\log q_{k,m}\log\log q_{k,m} \ge \frac1b m^2 (\log \|\vu\|_\infty)^2
$$
}

Finally, by taking square root in the both sides of~\eqref{proposition_db_implication_1}, \eqref{proposition_db_implication_2}, \eqref{proposition_db_implication_3} and~\eqref{proposition_db_implication_4} and summing up the results we find
\hidden{
taking the mean value of three estimates on $\log
q_{k,m}\log\log q_{k,m}$ we have that
}
\begin{equation}\label{eq_prop1}
\begin{array}{rl}
C\sqrt{\logc q_{k,m}\logc\logc q_{k,m}}\!\!\! &\geq \logc(k+1) + k\logc k + m\logc d\\
&+ (m+(2k+1)(\log_d(2k+1)+1))\logc\|\vu\|_\infty,
\end{array}
\end{equation}
%where $C = 2\sqrt{2}+2\sqrt{5\cdot \log\|\vu\|_{\infty}}+2$.
where the constant $C$ is defined by~\eqref{def_C}. Finally, by taking the
exponents base two from both sides of~\eqref{eq_prop1}, we
find~\eqref{eq_prop_first}, hence derive~\eqref{proposition_db_claim_2}.
\end{proof}

\begin{remark}
Note that the constant $C$ in Proposition~\ref{proposition_db} is rather far
from being optimal. The proof above can be significantly optimized to reduce
its value. However that would result in more tedious computations. All one
needs to show is the inequality~\eqref{eq_prop1}.
\end{remark}

\section{Proof of Theorem~\ref{th_main}}
%{Approximation properties of $g_\vu(b)$ and $f(b)$.}

%Throughout this section, we use the
%following definition of the constant $C$, implied by the statement of Proposition~\ref{proposition_db}:
%$$
%C:=2\sqrt{2}+2\sqrt{5\cdot\log\|\vu\|_{\infty}}+2.
%$$
%constant $C$ defined by~\eqref{def_C}.
We will prove the folowing result.

\begin{theorem} \label{thm_im}
Let $b\geq 2$. There exists an effectively computable constant $\gamma$,
which only depends on $d$ and $\vu$, such that for any $p\in\Z$ and any
sufficiently large $q\in\N$,
%\begin{equation} \label{thm_im_q_lb_absolute}
%q\geq 1523, %\max\left(6561,b^{b^b}\right),%b^{b^{2^9}}, %\max\left(b^4,b^b\right)
%\end{equation}
we have
\begin{equation} \label{thm_im_result}
\left|g_{\vu}(b)-\frac{p}{q}\right|\geq\frac{|g_\vu(b)|}{4b\cdot
q^{2}\cdot\exp\left(\gamma\sqrt{\logc q \logc \logc q}\right)}.
\end{equation}
\end{theorem}

It is easy to see that Theorem~\ref{th_main} is a straight corollary of
Theorem~\ref{thm_im}. Indeed, if $f(b)$ from Theorem~\ref{th_main} is not
zero then so is $g_\vu(b)$ and the lower bound~\eqref{thm_im_result} is
satisfied for all large
enough $q$, therefore %for any $\epsilon>0$
the inequality
$$
\left|g_\vu(b)-\frac{p}{q}\right| < q^{-2}\exp(-%(\log_2e\cdot
\gamma%+\epsilon)
\sqrt{\log q\log \log q})
$$
has only finitely many solutions. By definition, this implies that
$g_{\vu}(b)$ and in turn $f(b)$ are both $q^{-2}\exp(-%(\log_2e\cdot
\gamma%+\epsilon)
\sqrt{\log q\log \log q})$-badly approximable.

\begin{proof}[Proof of Theorem~\ref{thm_im}]%[Proof of Theorem~\ref{thm_im}]
%In all this proof we understand by $\log$ the binary $\log$, that is $\log=\log_2$.
In this proof, we will use the constant $C$ defined by~\eqref{def_C}. Fix a
couple of integers $p$ and $q$. We start with some preliminary calculations
and estimates.

Define $x>2$ %, where $e$ denotes the base of the natural logarithm,
to be the unique solution of the following equation
\begin{equation} \label{thm_im_def_x}
%$$
q=\frac{1}{12}\cdot x\cdot 2^{-\frac32C\sqrt{\logc x\logc\logc x}},
%$$
\end{equation}
where the constant $C$ is defined by~\eqref{def_C}.

The condition $x>2$ ensures that both $\logc x$ and the double logarithm $\logc\logc x$
exist and are positive, hence $2^{-C\sqrt{\logc x\logc\logc x}}<1$ and
thus
%Taking into account assumption~\eqref{thm_im_q_lb_absolute}, we have

%\frac{9}{2}
\begin{equation} \label{q_leq_x}
12 q < x.
\end{equation}
%Moreover, by using an evident consequence of assumption~\eqref{thm_im_q_lb_absolute}
For large enough $q$ we then have
$$
\frac{81}{4}C^2 \logc \logc x < \logc x
$$
and therefore
\begin{equation} \label{ie_loglog_x_leq_log_x}
2^{\frac32 C\sqrt{\logc x\logc \logc x}}<x^{1/3}.
\end{equation}
%(just because $\log x$ is positive, and so $\log \log x$ is well-defined, then we have the inequality above for all such real values of $x$).

From~\eqref{thm_im_def_x} and~\eqref{ie_loglog_x_leq_log_x} we readily infer
\begin{equation} \label{thm_im_preliminary_ub_x}
x<\left(%\frac{9}{2}
12 q\right)^{3/2},
\end{equation}
Rewrite~\eqref{thm_im_def_x} in the following form
\begin{equation} \label{thm_im_def_x_2}
x=%\frac{9}{2}
12 q\cdot 2^{\frac32C\sqrt{\logc x\logc\logc x}}.
\end{equation}
Then, by applying~\eqref{thm_im_preliminary_ub_x} to it we find that, %by taking into account the upper bound~\eqref{thm_im_preliminary_ub_x},
for large enough $q$,
\begin{equation} \label{thm_im_ub_x}
x<%\frac{9}{2}
12q\cdot 2^{2C\sqrt{\logc q\logc\logc q}}.
\end{equation}
%(to infer the inequality~\eqref{thm_im_ub_x}, we use the consequences of the assumption~\eqref{thm_im_q_lb_absolute} like $\frac{9}{2}q<q^{5/4}$ and $\frac{15}{8}\log q < \left(\log q\right)^{4/3}$).

Denote
\begin{equation} \label{def_t}
t:=\log_b x.
\end{equation}
Fix an arbitrary value $\tau\ge \tau_0>1$, where $\tau_0 = \tau_0(\vu)$ is a
parameter which only depends on $\vu$ and which we will fix later (namely, it
has to ensure inequality~\eqref{cond_tau}). Assume that $t>2$ is large
enough (that is, assume $q$ is large enough, then by~\eqref{q_leq_x} $x$ is
large enough, hence by~\eqref{def_t} $t$ is large enough), so that
\begin{equation} \label{t_log_t}
%\max\left(m\logc d,d\right)
d \leq
\frac1\tau\sqrt{\frac{t}{\logc t}}.
\end{equation}
%Note that~\eqref{t_log_t} and $t>2$ implies, in particular,
As $t>2$, we also have $t\logc t > 2$. Choose an integer $n$ of the form
$n:=k\cdot d^m$ with $m\in\N$, $k\in\Z$ such that
\begin{eqnarray}
t\leq & n & \leq t+d\tau\sqrt{t\logc t}, \label{thm_im_distance_t_n_ub}\\
\tau \sqrt{t\logc t} \leq & d^m & \leq d \tau \sqrt{t\logc t}. \label{thm_im_2m_lb_1}
\end{eqnarray}
One can easily check that such $n$ always exists.

Inequalities \eqref{t_log_t}, \eqref{thm_im_distance_t_n_ub}
and~\eqref{thm_im_2m_lb_1} imply
\begin{equation} \label{thm_im_k_ub_1}
k=\frac{n}{d^m}\leq\frac{t+d\tau \sqrt{t\cdot\logc
t}}{\tau\sqrt{t\cdot\logc t}}=\frac1\tau\sqrt{\frac{t}{\logc t}}+d\le
\frac2\tau\sqrt{\frac{t}{\logc t}}.
\end{equation}
%By using the assumption~\eqref{thm_im_q_lb_absolute}
%\hidden{
%$$
%t\geq b^4,
%$$
%then
%}
%we can infer from~\eqref{thm_im_k_ub_1} a simpler looking upper bound,
%\begin{equation} \label{thm_im_k_ub_2}
%k\leq\frac23\sqrt{\frac{t}{\log t}}.
%\end{equation}
Then we deduce, for $t$ large enough,
$$
k\logc k\le \frac2\tau\sqrt\frac{t}{\logc t} \left(\logc
(2/\tau)+\frac12\logc t-\frac12\logc\logc t\right)< \frac2\tau
\sqrt{t\logc t}.
$$
Therefore, for any $\tau$ large enough, that is for any $\tau\geq\tau_0$, where $\tau_0$ depends only on $\vu$, we have
\begin{multline} \label{cond_tau}
2+\logc (k+1)+\frac{k}{2}\logc k %+ m\logc d
\\+ (k+1)(\log_d(2k+1)+1)\logc
\|\vu\|_\infty)
<\tau \sqrt{t\logc t}.
\end{multline}
By taking the exponent base two of the left hand side of~\eqref{cond_tau} and
the exponent base $b\geq 2$ of the right hand side of~\eqref{cond_tau}, and
then using~\eqref{thm_im_2m_lb_1}, we ensure that~\eqref{lem5_assump} is
satisfied. We can also take $q$ (and, consecutively, $t$) large enough so
that $m$, bounded from below by~\eqref{thm_im_2m_lb_1}, satisfies $d^m \ge
m^2 (\logc\|\vu\|_\infty)^2$, and then
necessarily~\eqref{proposition_db_assump} is verified. Also,
\eqref{thm_im_distance_t_n_ub} and~\eqref{thm_im_2m_lb_1} imply that, for $t$
large enough, $k\geq 2$.

Hence we have checked all the conditions on $k$ and $m$ from
Proposition~\ref{proposition_db}. It implies that the integers $p_{k,m}$ and
$q_{k,m}$, defined by~\eqref{def_p_k_m} and~\eqref{def_q_k_m}, satisfy
inequalities~\eqref{proposition_db_u_1} and~\eqref{proposition_db_claim_2}.
%The upper bound~\eqref{thm_im_k_ub_1} implies
%\begin{equation} \label{k_c_ub}
%\begin{aligned}
%3k\log k+\log 4&\leq 3\cdot\frac23\cdot\sqrt{\frac{t}{\log t}}\cdot\frac12\left(\log t - \log\log t\right)+\log 4\\
%&\leq
%\sqrt{t\cdot\log t},
%\end{aligned}
%\end{equation}
%where in the last inequality we have used the assumption~\eqref{thm_im_q_lb_absolute}, which clearly implies
%$$
%\log 4\leq \cdot\sqrt{\frac{t}{\log t}}\log\log t.
%$$
%By comparing~\eqref{thm_im_2m_lb_1} and~\eqref{k_c_ub}, we find
%$$
%2^m\geq 3k\log k+\log 4,
%$$
%hence
%\begin{equation} \label{thm_im_main_relation_2m_k}
%b^{2^m}\geq 4\cdot k^{3k}.
%\end{equation}
Lemma~\ref{lemma_ub_c_i} and inequality~\eqref{lem5_assump} imply the
inequality~\eqref{lemma_qkm_db_bdm_assumption}, so we can use
Lemma~\ref{lemma_qkm_db}, i.e. we have
\begin{equation} \label{thm_im_q_k_m_db}
\frac12 |a_{k,k}|b^n\leq q_{k,m}\leq\frac32 |a_{k,k}|b^n.
\end{equation}
In case if
$$
\frac{p}{q}=\frac{p_{k,m}}{q_{k,m}},
$$
the result~\eqref{thm_im_result} readily follows from the lower
bound~\eqref{proposition_db_u_1} in Proposition~\ref{proposition_db}.

\hidden{
We readily infer
$$
\left|q-q_{k,m}\right|\leq
$$
}

We proceed with the case
\begin{equation} \label{thm_im_case_different}
\frac{p}{q}\ne\frac{p_{k,m}}{q_{k,m}}.
\end{equation}

By triangle inequality, and then by the upper
bound~\eqref{proposition_db_claim_2}, we have
\begin{equation} \label{thm_im_main_lb_1}
\begin{aligned}
\left|g_\vu(b)-\frac{p}{q}\right|&\geq\left|\frac{p_{k,m}}{q_{k,m}}-\frac{p}{q}\right|-\left|g_\vu(b)-\frac{p_{k,m}}{q_{k,m}}\right|\\
&\geq \frac{1}{q_{k,m} q}-\frac{3\cdot 2^{C{\sqrt{\logc q_{k,m}\logc\logc
q_{k,m}}}}}{q_{k,m}^2}.
\end{aligned}
\end{equation}

By %Lemma~\eqref{lemma_qkm_db}, we have $q_{k,m}\le \frac32 |a_{k,k}|b^n$. Thus, by
applying the upper bound in~\eqref{thm_im_q_k_m_db} complimented with~\eqref{ineq_hq_2}, we find
$$
\logc q_{k,m} \leq \logc \frac32+ k/2\logc k+
k(\log_d(2k+1)+1)\logc\|\vu\|_\infty + n\logc b
$$
Upper bounds~\eqref{thm_im_k_ub_1} on $k$
and~\eqref{thm_im_distance_t_n_ub} on $n$ ensure that for large
enough $q$ we have
\begin{equation} \label{thm_im_ub_1}
2^{C\sqrt{\logc q_{k,m}\logc\logc q_{k,m}}}\leq 2^{\frac32C\sqrt{\logc
x\logc\logc x}}.
\end{equation}

The formula~\eqref{def_t} for $t$ and the lower bound
in~\eqref{thm_im_distance_t_n_ub} together give $b^n\geq x$. Then, by using
the lower bound~\eqref{thm_im_q_k_m_db},
%~\eqref{lemma_qkm_db_claim} in Lemma~\eqref{lemma_qkm_db},
we find
\begin{equation} \label{thm_im_lb_1}
q_{k,m}\geq \frac12 b^n\geq \frac12 x.
\end{equation}
%Next, by~\eqref{thm_im_distance_t_n_ub} and the lower bound
%$q_{k,m}\ge \frac12 b^n$ (see the lower bound~\eqref{lemma_qkm_db_claim} in Lemma~\eqref{lemma_qkm_db}) we finally get
By using the estimates~\eqref{thm_im_ub_1} and~\eqref{thm_im_lb_1} on the numerator and denominator respectively, and then by substituting the value of $x$ given by~\eqref{thm_im_def_x_2}, we find
$$
\frac{3\cdot 2^{C{\sqrt{\logc q_{k,m}\logc\logc q_{k,m}}}}}{q_{k,m}^2}\leq
\frac{3\cdot 2^{\frac32C{\sqrt{\logc x\logc\logc x}}}}{\frac12 x\cdot
q_{k,m}} \stackrel{\eqref{thm_im_def_x_2}}=\frac{1}{2 q_{k,m} q},
$$
hence, recalling~\eqref{thm_im_main_lb_1}, we find
\begin{equation} \label{thm_im_lb_12}
\left|g_\vu(b)-\frac{p}{q}\right|\geq\frac{1}{2 q_{k,m} q}.
\end{equation}

%THEN THE IM FOLLOWS FROM~\eqref{thm_im_ub_x}

By inequality~\eqref{thm_im_q_k_m_db} combined with the upper bound in~\eqref{thm_im_distance_t_n_ub} and then~\eqref{def_t} and~\eqref{thm_im_ub_x} we get that, for $q$ large enough,%~\eqref{thm_im_2m_lb_1}
\begin{equation*} %\label{thm_im_lb_2}
q_{k,m}\leq\frac32 |a_{k,k}|b^n\leq \frac32 |a_{k,k}|
b^{t+d\tau\sqrt{t\logc t}} \leq 18 |a_{k,k}|q\cdot
2^{(2d\tau\logc b+2C)\sqrt{\logc q \logc\logc q}}.
\end{equation*}
The bound~\eqref{ineq_hq_2} implies % ensures that $\logc |a_{k,k}|$ does not
\begin{equation} \label{ineq_hq_2_log_corollary}
\logc |a_{k,k}|\leq \frac{k}{2}\logc k+k(\log_d (2k+1) + 1)\logc\|\vu\|_\infty.
\end{equation}
%exceed
By comparing the right hand side of this inequality with the left hand side
in~\eqref{cond_tau} we find
%Therefore,~\eqref{cond_tau} implies
$$
|a_{k,k}|\leq 2^{2\tau\sqrt{\logc q\logc\logc q}}
$$
and then
$$
q_{k,m}\le 18 q\cdot 2^{(2d\tau\logc b+2\tau+2C)\sqrt{\logc q\logc\logc q}}
$$
Finally, \eqref{thm_im_lb_12} implies
$$
\left|g_\vu(b)-\frac{p}{q}\right|\geq\frac{1}{36 q^2\cdot
2^{(2d\tau\logc b+2\tau+2C)\sqrt{\logc q \logc\logc q}}}.
$$
This completes the proof of the theorem with $\gamma = \ln
2\cdot\left(2d\tau\logc b+2\tau+2C\right)$.
\end{proof}

%\begin{acknowledgement}
%Evgeniy Zorin acknowledges the support of EPSRC  Grant EP/M021858/1.
%\end{acknowledgement}


\begin{thebibliography}{0}

\bibitem{adamczewski_cassaigne_2006} B.~Adamczewski, J.~Cassaigne,
``Diophantine properties of real numbers generated by finite
automata'', Comp. Math., 142, 2006, 1351 -- 1372.

\bibitem{adamczewski_rivoal_2009} B.~Adamczewski, T.~Rivoal, ``Irrationality measures for some automatic real numbers'', Math. Proc. Cambridge Phil. Soc., 147 (2009), 659--678.

\bibitem{APWW1998} J.-P.~Allouche, J.~Peyri\`ere, Z.-X.~Wen and
Z.-Y.~Wen, ``Hankel determinants of the Thue-Morse sequence'', Ann.
Inst. Fourier (Grenoble), 48, 1998, 1-27.

\bibitem{AS2003} J.-P.~Allouche and J.~Shallit, ``Automatic Sequences: Theory, Applications, Generalizations''. Cambridge University Press, Cambridge, 2003.

\bibitem{badziahin_2017} D.~Badziahin, ``Finding special factors of values of polynomials at integer
points''. Preprint.

\bibitem{BZ2015} Badziahin, D. \and Zorin, E., ``Thue-Morse constant is not badly approximable''. International Mathematics Research Notices 2015(19): 9618-9637.

\bibitem{badziahin_zorin_2015} Badziahin, D. \and Zorin, E., ``On generalized Thue-Morse functions and their
values.''. Preprint, https://arxiv.org/abs/1509.00297.

\bibitem{becker_1994} P.~G.~Becker, ``k-regular power series and
Mahler-type functional equations''. J. Numb. Theor., 49, 1994, 269
-- 286.

\bibitem{bell_bugeaud_coons_2015} J.~P.~Bell, Y.~Bugeaud, M.~Coons,
    ``Diophantine approximation of Mahler numbers''. Proc. Lond. Math. Soc., 110(5), 2015, 1157 -- 1206.

\bibitem{bugeaud_2011} Y.~Bugeaud, ``On the rational approximation
to the Thue-Morse-Mahler numbers''. Ann. Inst. Fourier, 61, 2011,
2065 -- 2076.

\bibitem{bugeaud_han_wen_yao_2015} Y. Bugeaud, G. N. Han, Z. Y. Wen, J. Y
Yao, ``Hankel determinants, Pad\'{e} approximations and
irrationality exponents''. Int. Math. Res. Notes, 2016(5), 2016,
1467 -- 1496.

\bibitem{coons_2013} M.~Coons, ``On the rational approximation to
the sum of reciprocals of the Fermat numbers''. Ramanujan J., 30,
2013, 39 -- 65.

\bibitem{DNNS1996} D.~Duverney, Ke.~Nishioka, Ku.~Nishioka, I.~Shiokawa. ``Transcendence of Jacobi's theta
series'', Proc. Japan. Acad. Sci, Ser A, 72, (1996), 202-203.

\bibitem{guo_wu_wen_2014} Y.~J.~Guo, Z.~X.~Wen, W.~Wu, ``On the
irrationality exponent of the regular paperfolding numbers''. Lin.
Alg. Appl., 446, 2014, 237 -- 264.

\bibitem{HS1965} J.~Hartmanis and R.~E.~Stearns, ``On the Computational Complexity of Algorithms'', Transactions of the American Mathematical Society, Vol. 117 (May, 1965), pp. 285-306.

\bibitem{Liouville1851} J.~Liouville, ``Sur des classes tr\`es-\'etendues de quantit\'es dont la valeur n’est ni
alg\'ebrique, ni m\^eme r\'eductible \`a des irrationnelles alg\'ebriques'', J. de Math. Pures et Appl.,1851.

\bibitem{Mah1930} K.~Mahler, ``Arithmetische Eigenschaften einer Klasse transzendentaltranszendenter
Funktionen'', Math. Z. 32 (1930), 545-585.

\bibitem{Ni1996} K.~Nishioka, ``Mahler Functions and Transcendence'', Lecture Notes in Math. 1631, Springer,
1996.

\bibitem{vdP_S} A.~J.~van~der~Poorten, J.~O.~Shallit, ``Folded Continued Fractions'', J. Number Theory, 40, 1992, no.~2, 237--250.

\bibitem{poorten_1998} A. J. van der Poorten., ``Formal power series and their continued fraction
expansion''. Algorithmic Number Theory. Lecture notes in Computer
Science, 1423, pp 358 -- 371, Springer, Berlin, 1998.

\bibitem{wu_wen_2014} Z.~X.~Wen, W.~Wu, ``Hankel determinant of the
Cantor sequence'' (in Chinese), Scienta Sinica Math., 44, 2014, 1059
-- 1072.

%\bibitem{zorin_2017} E.~Zorin ``Algebraic independence and normality of the values of Mahler's functions'', preprint,  arXiv:1309.0105.

\end{thebibliography}
\end{document}